\documentclass{amsart}
\usepackage{amssymb, amsmath, tikz-cd}
\usepackage{stmaryrd}
\usepackage{xcolor}

\title{Solutions to difference equations have few defects}
\author{Patrick Ingram}
\address{York University, Toronto, Canada}
\email{pingram@yorku.ca}
\date{\today}
\thanks{This research is supported by a grant from NSERC of Canada. The author would like to thank an anonymous referee for a careful reading, and many helpful comments and corrections.}

\renewcommand{\epsilon}{\varepsilon}
\renewcommand{\phi}{\varphi}

\newcommand{\PP}{\mathbb{P}}

\newcommand{\CC}{\mathbb{C}}
\newcommand{\RR}{\mathbb{R}}

\newcommand{\Res}{\operatorname{Res}}

\newcommand{\ord}{\operatorname{ord}}

\newtheorem{lemma}{Lemma}

\newtheorem{theorem}{Theorem}

\theoremstyle{definition}
\newtheorem{remark}{Remark}

\begin{document}
\begin{abstract}
We establish a strong form of Nevanlinna's Second Main Theorem for solutions to difference equations
\[f(z+1)=R(z, f(z)),\]
with the coefficients of $R$ growing slowly relative to $f$, and $\deg_w(R(z, w))\geq 2$.
\end{abstract}
\maketitle

Picard famously showed that a nonconstant meromorphic function $f:\CC\to \widehat{\CC}$ omits at most two values, and Nevanlinna generalized this with his Second Main Theorem. Specifically, at each $a\in\widehat{\CC}$ the Nevanlinna defect satisfies $0\leq \delta_f(a)\leq 1$, and has $\delta_f(a)=1$ for any omitted value $a$, and it follows from Nevanlinna's result that the defects of $f$ at all points of $\widehat{\CC}$ sum to at most 2.

At least since work of Yanagihara~\cite{MR621536, MR634897}, there has been interest in applications of Nevanlinna Theory to difference equations of the form
\begin{equation}\label{eq:eq}
f(z+1)=R(z, f(z)),
\end{equation}
 where $R(z, w)$ is a rational function in $w$, with coefficients meromorphic in $z$. A relation of the form~\eqref{eq:eq} makes it harder for $f$ to omit values, and Yanagihara~\cite{MR621536} showed that such solutions omit no values, except in some special circumstances, if $R(z, w)=R(w)\in\CC(w)$ is constant in $z$. 
 
 Our main result is a strong form of Nevanlinna's Second Main Theorem for solutions to~\eqref{eq:eq}, and generalizes Yanagihara's result from the case of constant coefficients, and from omitted values to defects (although see also~\cite{MR2106974}, which gives a similar generalization when the coefficients are constant). We will say that $a\in \CC$ is \emph{shift-exceptional} for $R(z, w)$ if $w=a$ is a totally ramified fixed point of $w\mapsto R(z, R(z-1, w))$, and that $a\in\CC$ is \emph{ordinary} otherwise (note that most $R$ have no exceptional points). See Section~\ref{sec:prelim} for the definitions of the Nevanlinna characteristic function $T_f(r)$, the proximity $m_f(r; a)$, and the defect $\delta_f(a)$.

\begin{theorem}\label{th:defect}
Let $f$ be a solution to~\eqref{eq:eq}, suppose that $\deg_w(R)\geq 2$, that the coefficients of $R$ are slow-growing relative to $f$, and that $a\in \CC$ is ordinary for $R$. Then $\delta_f(a)= 0$ and in fact \[m_f(r; a)=o(T_f(r))\]  as $r\to\infty$
outside of some set of finite measure.
\end{theorem}
Note that the Second Main Theorem gives \[\sum_{i=1}^km_f(r; a_i)\leq (2+o(1))T_f(r)\] as $r\to\infty$ outside of a set of finite length, for distinct $a_1, ..., a_k\in\CC$, and so Theorem~\ref{th:defect} can also be seen as a strengthening of this in the special case of solutions to~\eqref{eq:eq}. If $R$ has two exceptional points (the most possible) then Theorem~\ref{th:defect} reduces to the usual Second Main Theorem if we include the exceptional points in our sum, but in this case $R$ is (up to change of variables over the field of coefficients) either $R(z, w)=w^d$ for $d=\pm \deg_w(R)$. The case of one exceptional point coincides with $R(z, w)$ being a polynomial in $w$ (again up to change of coordinates), and so in the general case we may apply Theorem~\ref{th:defect} to any $a\in\CC$. 
Also, just as the Second Main Theorem can be extended to the case in which the $a_i$ are functions of slow growth (relative to $f$), we note that Theorem~\ref{th:defect} also holds in the context of such moving targets, as long we assume that $a(z)$ and $a(z+1)$ both grow slowly.

The proof is motivated by a result of Silverman~\cite{MR1240603} on diophantine approximation in arithmetic dynamics. As noted above, this idea was already used by Ru and Yi~\cite{MR2106974} to produce similar results, under different assumptions.


\section{Notation and background}\label{sec:prelim} 
We write $M$ for the field of meromorphic functions on $\CC$, and $\overline{M}$ for the algebraic closure, the field of algebroid functions on $\CC$. Given $\beta\in\CC$, we write $\ord_{\beta}(f)$ for the order of vanishing of $f(z)$ at $z=\beta$, and set
\[\ord_\beta^+(f)=\max\{\ord_\beta(f), 0\}.\] We note the convenient property that \[\ord_\beta(f+g)\geq \min\{\ord_\beta(f), \ord_\beta(g)\},\] with equality except perhaps when $\ord_\beta(f)=\ord_\beta(g)$.
We then set, as usual,
\[n_f(r)=\sum_{|z|\leq r}\ord_z^+(1/f),\]
and, with $r>0$,
\begin{multline*}
N_f(r)=\int_0^r (n_f(t)-n_f(0))\frac{dt}{t}+n_f(0)\log r\\=\sum_{0<|z|\leq r}\ord_z^+(1/f)\log\frac{r}{|z|}+\ord_0^+(1/f)\log r,	
\end{multline*}
for the \emph{Nevanlinna counting function}. The \emph{proximity function} is
\[m_f(r)=\int_0^{2\pi}\log^+|f(re^{i\theta})|\frac{d\theta}{2\pi},\]
and the \emph{characteristic function} is
\[T_f(r)=N_f(r)+m_f(r).\]
We will also set $N_f(r; a)=N_{1/(f-a)}(r)$ for any $a\in M\setminus\{f\}$, and similarly for $m$ and $T$, while $N_f(r; \infty)=N_f(r)$.

The First Main Theorem of Nevanlinna~\cite{MR1555200} gives
\[T_f(r; a)=T_f(r)+O_a(1)\]
for any $a\in \CC$, or more generally $T_f(r; a)=T_f(r)+O(T_a(r))$ for any $a\in M\setminus\{f\}$.

As in Steinmetz~\cite{MR850619}, we write $K_f$ for the field of functions of type $S(f, r)$, that is, the field of $g\in M$ such that $T_g(r)=o(T_f(r))$, except possibly on a set of finite Lebesgue measure.

We define also the \emph{Nevanlinna defect} of $f$ at $a\in K_f$ by
\[\delta_f(a)=\liminf_{r\to\infty}\frac{m_f(r; a)}{T_f(r)},\]
which satisfies $0\leq \delta_f(a)\leq 1$ by definition. Note that if $f(z)=a$ has no solutions, then $N_f(r; a)\equiv 0$, and so $m_f(r;a ) = T_f(r)+S(f, r)$, whence $\delta_f(a)=1$. 

%


\section{Some technical lemmas}

We maintain the notation and conventions of the previous section. Our first lemma relates the proximity function of a rational function of $f$ to the denominator of the rational function, and the proximity of $f$ itself.

Our first lemma is a basic fact from commutative algebra.

\begin{lemma} 
\label{lem:res}
Let $F$ be a field, and let $P, Q\in F[T]$ be polynomials of degree $d$ and $e$, with no common factor. Then there exist polynomials $A, B\in F[T]$ of degree (at most) $e-1$ and $d-1$ such that 
\[A(T)P(T)+B(T)Q(T)=1.\]
Furthermore, if $\Res(P, Q)$ is the resultant of $P$ and $Q$, then $\Res(P, Q)$, the coefficients of $\Res(P, Q)A(T)$, and the coefficients of $\Res(P, Q)B(T)$ are all polynomials in the coefficients of $P$ and $Q$. 
\end{lemma}

\begin{proof}
With $a_{e-1}, ..., a_0, b_{d-1}, ..., b_0$ as indeterminates, consider
\[S(T)=(a_{e-1}T^{e-1}+\cdots +a_0)P(T)+(b_{d-1}T^{d-1}+\cdots +b_0)Q(T),\]
which has degree $d+e-1$. Setting $S(T)=0$ as polynomials, and equating coefficients, yields a system of $d+e$ linear equations in $d+e$ unknowns; $\Res(P, Q)$ is the determinant of the coefficient matrix.

If $\Res(P, Q)=0$, then there is a non-trivial solution to $S(T)=0$, giving an equality
\[\frac{P(T)}{Q(T)}=-\frac{b_{d-1}T^{d-1}+\cdots +b_0}{a_{e-1}T^{e-1}+ ...+ a_0}\]
(or the reciprocals) which contradicts the degrees of $P$ and $Q$ (unless they have a common factor).

Given that $\Res(P, Q)\neq 0$, we can then solve $S(T)=1$ as a system of linear equations in the $a_i$ and $b_j$, using Cramer's Rule. In paricular, $\Res(P, Q)$, the coefficients of $\Res(P, Q)A$, and the coefficients of $\Res(P, Q)B$ are all determinants of matrices (of side length $d+e$) whose entries are coefficients of $P$ and $Q$. It follows that these quantities are polynomials in the coefficients of $P$ and $Q$.
\end{proof}

The following is due to Valiron~\cite{MR1504970} in a special case, and Mohon'ko~\cite{MR0298006} more generally, and is straightforward to prove from Lemma~\ref{lem:res}. 
\begin{lemma}[Valiron~\cite{MR1504970}, Mohon'ko~\cite{MR0298006}]\label{lem:valiron}
	Let $R(z, w)\in K_f(w)$, and let $f\in M$. Then
	\[T_{R(z, f(z))}(r)=\deg_w(R)T_f(r)+S(f, r).\]
\end{lemma}

\begin{remark}	
Yanagihara~\cite{MR621536, MR634897} showed that if $R(z, w)$ is rational in both variables, and $\deg_w(R)\geq 2$, then any finite-order solution to $f(z+1)=R(z, f(z))$ is  rational. 
This raises the question of describing all rational solutions, and we refer the reader here to a recent survey~\cite{MR3650542} and the author's related note~\cite{diff_fin}. The referee points out to the author that 
\end{remark}

For convenience in the next lemma, we suppress the notational dependence on~$z$.

\begin{lemma}\label{lem:mvt}
Let  $P(w), Q(w)\in K_f[w]$ be polynomials, with $\deg_w(P)\geq \deg_w(Q)$. Then for $f$ with $Q(f)\not\equiv 0$,
\[m_{P(f)/Q(f)}(r)\leq (\deg_w(P)-\deg_w(Q))m_f(r)+m_{Q(f)}(r; 0)+S(f, r).\]	
\end{lemma}

\begin{proof}
Let $\beta\in\CC$, let $S(w)=c_mw^m+\cdots +c_0\in M[w]$, and set 
\begin{equation}\label{eq:basin}C_\beta(S)=\max_{0\leq i<m}\left\{\frac{\ord^+_\beta\left(c_m/c_i\right)}{m-i}, \frac{\ord_\beta^+(c_m^{-1})}{m}\right\}\geq 0.\end{equation}
Then for $f$ with $\ord_\beta(1/f)>C_\beta(S)$, we have
\[\ord_\beta(c_mf^m)<\ord_{\beta}(c_if^i)\]
for all $0\leq i<m$, and so
\begin{equation*}	
\ord_\beta(1/S(f))=m\ord_\beta(1/f)-\ord_\beta(c_m)>0.
\end{equation*}
Note, on the other hand, that for any $C\geq 0$, the hypothesis $\ord_{\beta}(1/f)\leq C$ always implies \[\ord_{\beta}(1/S(f))\leq mC+\kappa_{\beta, S},\] with \[\kappa_{\beta, S}=\sum_{i=0}^m\ord_\beta^+(1/c_i)
\geq 0.\]

Suppose first that
\[\ord_\beta(1/f)>C_\beta(P)+C_\beta(Q).\]
Writing $P(w)=a_dw^d+\cdots +a_0$ and $Q(w)=b_ew^e+\cdots +b_0$,
we have
\begin{align}
\ord_{\beta}(Q(f)/P(f))&=\ord_{\beta}(Q(f))-\ord_{\beta}(P(f))\nonumber\\
&=-\deg(Q)\ord_{\beta}(1/f)+\ord_\beta(b_e)+\deg(P)\ord_{\beta}(1/f)-\ord_\beta(a_d)\nonumber\\
&\geq  (\deg(P)-\deg(Q))\ord^+_\beta(1/f)-\ord^+_\beta(b_e^{-1})-\ord^+_\beta(a_d).\label{eq:escape}
\end{align}
We also have $\ord_\beta^+(Q(f))\leq \ord_\beta^+(b_e)$ in this case,
and so
\begin{multline}\label{eq:PQest1}
\ord^+_{\beta}(Q(f)/P(f))\geq (\deg(P)-\deg(Q))\ord_{\beta}^+(1/f)+\ord^+_{\beta}(Q(f))\\-(\ord^+_\beta(b_e)+\ord^+_\beta(a_d)+\ord_\beta^+(b_e^{-1})).	
\end{multline}

Suppose, on the other hand, that we have $\ord_\beta(1/f)\leq C_\beta(P)+C_\beta(Q)$.
Now, by Lemma~\ref{lem:res} we can find $A(w), B(w)\in K_\phi[w]$ of degrees $e-1$ and $d-1$, respectively, such that
\[A(f)P(f)+B(f)Q(f) = 1.\] Note that if \[\ord_{\beta}(Q(f))> (d-1)(C_\beta(P)+C_\beta(Q))+\kappa_{\beta, B}\geq\ord_\beta(1/B(f)),\]
then $\ord_{\beta}(A(f)P(f))=0$, and so
\[\ord_\beta(P(f))=\ord_\beta(1/A(f))\leq (e-1)(C_\beta(P)+C_\beta(Q))+\kappa_{\beta, A}.\]
This then gives
\[\ord_{\beta}(Q(f)/P(f))\geq \ord_{\beta}(Q(f))-(e-1)(C_\beta(P)+C_\beta(Q))-\kappa_{\beta, A}\]
and hence
\begin{multline}\label{eq:PQest2}
\ord^+_\beta(Q(f)/P(f))\geq (\deg(P)-\deg(Q))\ord_{\beta}^+(1/f)+\ord_{\beta}^+(Q(f))\\-\kappa_{\beta, A}-(d-1)(C_\beta(P)+C_\beta(Q)).
\end{multline}

Finally, if we have $\ord_{\beta}(Q(f))\leq (d-1)(C_\beta(P)+C_\beta(Q))+\kappa_{\beta, B}$ and still $\ord_\beta(1/f)\leq C_\beta(P)+C_\beta(Q)$, we immediately have
\begin{multline}\label{eq:PQest3}
\ord^+_\beta(Q(f)/P(f))\geq 0\geq (\deg(P)-\deg(Q))\ord_{\beta}^+(1/f)+\ord_{\beta}^+(Q(f))\\-(2d-e-1)(C_\beta(P)+C_\beta(Q))-\kappa_{\beta, B}.
\end{multline}
Combining~\eqref{eq:PQest1}, \eqref{eq:PQest2}, and \eqref{eq:PQest3} we have in any case
\begin{multline}\label{eq:PQest4}
\ord^+_\beta(Q(f)/P(f))\geq 0\geq (\deg(P)-\deg(Q))\ord_{\beta}^+(1/f)+\ord_{\beta}^+(Q(f))-E_\beta\end{multline}
for
\begin{multline}\label{eq:Ebeta}E_\beta=(2d-e-1)(C_\beta(P)+C_\beta(Q))+\kappa_{\beta, B} \\+\kappa_{\beta, A}  +(\ord^+_\beta(b_e)+\ord^+_\beta(a_d)+\ord_\beta^+(b_e^{-1}))
\end{multline}
by the non-negativity of the various terms in the error.

Now, referring back to Lemma~\ref{lem:res}, not that every coefficient of $\operatorname{Res}(P, Q)A$ is a polynomial in the coefficients of $P$ and $Q$ of at most degree $(2d-1)!$, and so
\[\kappa_{\beta, A}\leq (2d-1)!\left(\sum_{i=0}^d\ord_\beta^+(1/a_i)+\sum_{i=0}^e\ord_\beta^+(1/b_i)\right)+\ord^+_\beta(\operatorname{Res}(P, Q)),\]
and similarly for $\kappa_{\beta, B}$,
and hence (from this and the  definitions of $C_\beta(P)$ and $C_\beta(Q)$)
\begin{align*}
\sum_{|\beta|\leq r}E_\beta\log\frac{r}{|\beta|}&\leq 2N_{\operatorname{Res}(P, Q)}(r; 0)\\&\quad +O\left(\sum m_{a_i}(r)+\sum m_{b_i}(r)+m_{a_d}(r; 0)+m_{b_e}(r; 0)\right) 	\\
&\leq  2T_{\operatorname{Res}(P, Q)}(r)+S(f, r)\\
&=S(f, r),
\end{align*}
since $\operatorname{Res}(P, Q)$ is itself a polynomial in the coefficients of $P$ and $Q$.
From this and~\eqref{eq:PQest4} we have
\[N_{Q(f)/P(f)}(r)\geq (\deg(P)-\deg(Q))N_f(r)+N_{Q}(r; 0)+S(f, r),\]
which in turn gives
\begin{align*}
m_{Q(f)/P(f)}(r)&=T_{Q(f)/P(f)}(r)-N_{Q(f)/P(f)}(r)\\
&\leq \deg(P/Q)T_f(r) -(\deg(P)-\deg(Q))N_f(r)-N_{Q(f)}(r; 0)+S(f, r)\\
&=(\deg(P)-\deg(Q))T_f(r)-(\deg(P/Q)-\deg(Q))N_f(r)+T_{Q}(r; 0) \\&\quad -N_{Q(f)}(r; 0)+S(f, r)\\
&=(\deg(P)-\deg(Q))m_f(r) +m_{Q(f)}(r; 0)+S(f, r).
\end{align*}
\end{proof}

The following lemma is closely related to a lemma of Silverman~\cite{MR1240603} (see also results of Ru and Yi~\cite{MR2106974}), but it is sufficiently different that we present a self-contained proof. As it becomes slightly more convenient in the next lemma, we will set $R_z(w)=R(z, w)$, thought of as a function of $w$ alone, so that $R(z, R(z-1, w))$ is written $R_z\circ R_{z-1}(w)$.

\begin{lemma}\label{lem:spread}
Suppose that $R_{z}\circ R_{z-1}(w)$ is not a polynomial in $w$, and write \[R_{z}\circ \cdots \circ R_{z-k+1}(w)=\frac{P_{k}(w)}{Q_{k}(w)}\]  in lowest terms (where we suppress the dependence on $z$). Further write \[Q_{k}(w)=\prod_{i=1}^{m_k}H_{i, k}(w)^{e_{i, k}}\] with the $H_{i, k}$ irreducible and distinct, and set $e_{0, k}=\deg_w(R)^k-\deg_w(Q_{k})$. Then we have
\[e_k:=\max\{e_{0, k}, e_{1, k}, ..., e_{m_k, k}\}=o(\deg_w(R)^k)\]
as $k\to \infty$.
\end{lemma}

\begin{proof} Let $\epsilon>0$, and set  $d=\deg_w(R)$. On the assumption that $R_{z+1}\circ R_z(w)$ is not a polynomial, we will show that $e_k\leq \epsilon d^k$ once $k$ is larger than some explicit value depending on $\epsilon$ and $R$.

Over the algebraic closure, let $g_i$ be chosen so that $g_0=\infty$, and $R_{z-i}(g_{i+1})=g_i$. Then the ramification index of $R_z\circ\cdots \circ R_{z-k+1}$ at $g_k$ is one of the $e_{j, k}$, and all $e_{j, k}$ are obtained with some such choice. Thus we are interested in bounding the ramification index
\[e_{R_z\circ\cdots \circ R_{z-k+1}}(g_k)=e_{R_z}(g_1)\cdots e_{R_{z+k-1}}(g_k),\]
independent of the choice of $g_k$.

Let $\sigma$ be the field automorphism $f^\sigma(z):=f(z+1)$ of $M$, which extends to the algebraic closure $\overline{M}$ and to the projective line over $\overline{M}$ (by $\infty^\sigma=\infty$.) Note that $e_{R_z^\sigma}(h)=e_{R_z}(h^{-\sigma})$ and so, in particular, we are interested in computing
\[\prod_{i=1}^{k} e_{R_z}(g_i^{\sigma^{i-1}}),\]
and it will suffice to show that this is $o(d^k)$ as $k\to\infty$.

Suppose first that there is no $j>0$ such that $g_j=\infty$. We claim that then the values $g_j^{\sigma^{j-1}}$ are distinct. If not, then there exist $j$ and $m$ with $j\geq m>0$ so that $g_j^{\sigma^{j-1}}=g_{j-m}^{\sigma^{j-m-1}}$, or $g_{j-m}=g_j^{\sigma^{m}}$. It then follows that
\begin{align*}
g_{j-2m}&= R_{z-(j-2m)}\circ \cdots \circ R_{z-(j-m-1)} (g_{j-m})\\
&= R_z^{\sigma^{-(j-2m)}}\circ \cdots \circ R_z^{\sigma^{-(j-m-1)}} (g_{j}^{\sigma^m})\\
&= \left(R_z^{\sigma^{-(j-m)}}\circ \cdots \circ R_z^{\sigma^{-(j-1)}} (g_{j})\right)^{\sigma^{m}}\\
&= g_{j-m}^{\sigma^m}\\
&= g_j^{\sigma^{2m}},
\end{align*}
and hence
\[g_{j-tm}=g_j^{\sigma^{tm}}\]
for all $0\leq t\leq j/m$.
It also then follows that
\[g_{j-tm-s}=R_z^{\sigma^{-(j-tm-s)}}\circ \cdots \circ R_z^{\sigma^{-(j-tm-1)}}(g_j^{\sigma^{tm}})=g_{j-s}^{\sigma^{tm}},\]
for any $s\geq 0$. In particular, writing $j=tm+s$ with $0\leq s<m$ and $t\geq 0$, we have
\[\infty = g_0 = g_{j-s}^{\sigma^{tm}},\]
and hence $g_{j-s}=\infty$, a contradiction because $j-s>j-m\geq 0$.

So in this case the $g_j^{\sigma^{j-1}}$ are distinct. By the Riemann-Hurwitz formula, and the fact that the arithmetic mean bounds the geometric mean, we have
\begin{align*}
e_{R_z\circ\cdots \circ R_{z-k+1}}(g_k)&=\prod_{i=1}^{k} e_{R_z}(g_i^{\sigma^{i-1}})\\
&\leq \left(\frac{\sum_{i=1}^ke_{R_z}(g_i^{\sigma^{i-1}})}{k}\right)^k\\
&\leq \left(1+\frac{\sum_{i=1}^k(e_{R_z}(g_i^{\sigma^{i-1}})-1)}{k}\right)^k\\
&\leq \left(1+\frac{2d-2}{k}\right)^k \\
&\leq  e^{2d-2}\\
&\leq \epsilon d^k
\end{align*}
as soon as $k\geq (\log \epsilon^{-1}+2d-2)/\log d$.

Now suppose that there is some $j> 0$ such that $g_j=\infty$, let $m$ be the least such value of $j$. Note that we then have $R_z\circ\cdots \circ R_{z-(m-1)}(\infty)=\infty$, and so $m$ depends only on $R$. Write
\[E=e_{R_{z}\circ \cdots\circ R_{z-(m-1)}}(\infty)^{1/m},\]
and assume for now that $E<d$, so in fact $E\leq (d^m-1)^{1/m}$. Let $t\leq k/m$ be the largest value with $g_{tm}=\infty$, and $s=k-tm$. If $s=0$, then
\begin{align*}
e_{R_z\circ \cdots \circ R_{z-k+1}}(g_k)&=\prod_{u=0}^{t-1}e_{R_{z-um}\circ \cdots \circ R_{z-(u+1)m+1}}(\infty)\\
&=\prod_{u=0}^{t-1}e_{R_{z}\circ \cdots \circ R_{z-m+1}}(\infty^{\sigma^{um}})\\
&=E^{k}\\
&\leq (d^m-1)^{k/m}\\
&\leq \epsilon d^k	
\end{align*}
as soon as \[k\geq m\log\epsilon/\log(1-1/d^m).\]

Otherwise, by the argument above, the values $g_j^{\sigma^{j-1}}$ are distinct for $tm\leq j\leq tm+s$, and so
\[e_{R_z\circ \cdots \circ R_{z-k+1}}(g_k)\leq E^{tm}\left(1+\frac{2d-2}{s}\right)^{s}\leq E^{tm}e^{2d-2}.\]
Note that $tm\leq k$, and so
\[\log \left(E^{tm}\left(1+\frac{2d-2}{s}\right)^s\right)\leq \frac{k}{m}\log (d^m-1) +(2d-2) \leq  \log (\epsilon d^k)\]
as soon as 
\[k\geq \frac{\log \epsilon^{-1}+(2d-2)}{\log d-\frac{1}{m}\log(d^m-1)}.\]

We are left with the case that $E=d$, or in other words 
\[d^m=e_{R_{z}\circ \cdots R_{z-(m-1)}}(\infty)=e_{R_z}(g_1)\cdots e_{R_z}(g_{m}^{\sigma^{m-1}}).\]
Since $1\leq e_{R_z}(w)\leq d$ for all $w\in \overline{M}$, and this value is attained at most twice (by the Riemann-Hurwitz formula), we have $m=1$ or $m=2$. If $m=1$, then $R_z$ is a polynomial, and if $m=2$ then $R_{z}\circ R_{z-1}$ is.

In any case, as long as $R_{z}\circ R_{z-1}(w)$ is not a polynomial in $w$, we have shown that $\max\{e_0, ..., e_m\}\leq \epsilon d^k$, and since $\epsilon>0$ was arbitrary, we are done.
\end{proof}

Finally, we give a standard characterization of the case in which $R_z\circ R_{z-1}(w)$ is a polynomial in $w$.

\begin{lemma}\label{lem:exceptional}
Suppose that $R_{z}\circ R_{z-1}(w)$ is a polynomial in $w$. Then either $R_z(w)$ is a polynomial in $w$, or \[R_z(w)=a_z(w-b_z)^{-d}+b_{z+1},\] for some $a, b\in M$.	
\end{lemma}

\begin{proof}
Suppose that $R_z(w)$ is not a polynomial in $w$. In $\overline{M}$, we may then choose some $a_z\neq \infty$ with $R_z(a_z)=\infty$. On the other hand, since $R_z\circ R_{z-1}(w)$ is a polynomial, any preimage of $\infty$ by $R_{z}$ must be equal to $R_{z-1}(\infty)$, and so $R_{z-1}(\infty)=a_z$, and in particular $a_z$ is the unique solution to $R_z(w)=\infty$. It follows that $a_z$ is in the field generated by the coefficients of $R_z(w)$.

Set $\mu_z(w)=w+a_{z+1}$, and $\mu_z^{-1}(w)=w-a_{z+1}$, and write $S_z=\mu_z^{-1}\circ R_z\circ \mu_{z-1}$.
We have
\[S_z(0)=R_z(0+a_z)-a_{z+1}=\infty,\]
while
\[S_z(\infty)=R_z(\infty)-a_{z+1}=0,\]
and so the rational function $S_z(w)$ of degree $\deg_w(R)$ must have the form $g_z w^{-\deg_w(R)}$ for some $g_z$ in the field generated by the coefficients of $R_z(w)$, and hence
\[R_z(w)=g_z(w-a_z)^{-\deg_w(R)}+a_{z+1}.\]
\end{proof}

%
%
%

\begin{remark}
Let $F$ be a field and	let $\sigma$ be an automorphism of $F$; that is, let $(F, \sigma)$ be a difference field. For $R(w)\in F(w)$, one might consider the dynamical system on the projective line $\PP^1_F$ given by $z\mapsto R(z)^\sigma$. As an example with $F=M$ and $f(z)^\sigma=f(z-1)$, fixed points of this dynamical system are precisely solutions to the difference equation~\eqref{eq:eq}. The proof of Lemma~\ref{lem:exceptional} in this context gives a classification of rational functions admitting \emph{exceptional points} (points with finite grand orbit under the difference-dynamical system), which reduces to the well-known statement in the theory of iteration of rational functions when $\sigma$ is trivial. Similarly, the proof of Lemma~\ref{lem:spread} generalizes to this context, too, and this generalization specializes to the lemma of Silverman~\cite{MR1240603} by taking $\sigma$ again to be the trivial automorphism.
\end{remark}

\section{The proof of the main theorem}

Finally, we cite a theorem of Steinmetz~\cite{MR850619}, to set the stage for the proof of Theorem~\ref{th:defect}.
\begin{theorem}[{Steinmetz~\cite[Satz 2]{MR850619}}]\label{th:steinmetz}
Let $H(w)\in K_f[w]$ have distinct roots. Then for any $\delta>0$ there exists a set $ E_\delta\subseteq \RR^+$ of finite measure such that, for $r\not\in E_\delta$ we have
\[m_f(r; \infty)+m_{H(f)}(r; 0)=m_f(r)+m_{1/H(f)}(r)\leq (2+\delta)T_f(r).\]
\end{theorem}

Wtih Steinmetz's version of the Second Main Theorem in hand, we may prove the main result.

\begin{proof}[Proof of Theorem~\ref{th:defect}]
First, we prove the theorem in the case $a=\infty$, under the hypothesis that $R_z(w)\in K_\phi(w)$ for $\phi(r)=o(T_f(r))$ as $r\to\infty$.

 Let $\epsilon>0$, and let $k\geq 0$ be an integer to be chosen later. Write \[R_z\circ \cdots \circ R_{z-k+1}(w)=\frac{P_{k}(w)}{Q_{k}(w)}\] in lowest terms as above, and again write $Q_{k}(w)=\prod H_{i, k}^{e_{i, k}}(w)$, where the $H_{i, k}$ are irreducible. Note that since $\phi$ is non-decreasing, $K_{\phi(r-1)}\subset K_{\phi}$, and in particular $P_k(w), Q_k(w), H_{i, k}(w)\in K_{\phi}[w]$. Setting
\[\epsilon_{i, j}=\begin{cases} 1 & i\leq j\\ 0&\text{otherwise,}\end{cases}\]
and 
\[J_{i, k}(w) = \prod_{i=1}^{m_k}H_{i, k}^{\epsilon_{i, e_{i, k}}}(w),\]
we see that $J_{i, k}$ has no repeated factors, and $Q_{k}(w) = \prod_{i=1}^{\max\{e_{1, k}, ..., e_{m_k, k}\}}J_{i, k}(w)$. For simplicity, we set $J_{i, k}(w)\equiv 1$ if $i>\max\{e_{1, k}, ..., e_{m_k, k}\}$, and we will also set $e_{0, k}=\deg_w(R)^k-\deg(Q_k)$, and write again $e_k=\max\{e_{0, k}, ..., e_{m_k, k}\}$. By Lemma~\ref{lem:spread}, we may choose $k$ sufficiently large so that
\begin{equation}\label{eq:bigk}e_k<\frac{\epsilon \deg_w(R)^k}{6}.\end{equation}

Note that $m_{st}(r)\leq m_s(r)+m_t(r)$ for any $s$ and $t$, and so we can bound the proximity function of $Q_k(f)$ in terms of those of $J_{i, k}(f)$, specifically as
\[m_{Q_k(f)}\leq \sum_{i=1}^{e_k}m_{J_{i, k}(f)}(r).\]
It follows from Lemma~\ref{lem:mvt}, Steinmetz's Theorem~\ref{th:steinmetz} with $\delta=1$, and the non-negativity of the proximity function
that
\begin{align}
m_{f}(r)&= m_{R_{z-1}\circ\cdots\circ R_{z-k}(f(z-k))}(r)  \nonumber\\ 
&\leq  (\deg(R_{z-1}\circ\cdots\circ R_{z-k})-\deg(Q_{k}))m_{f(z-k)}(r) +m_{Q_{k, z}(f(z-k))}(r; 0) \nonumber\\ &\quad +S(f, r) \nonumber \\
&\leq \sum_{i=1}^{e_k}\left(m_{f(z-k)}(r)+m_{J_{i, k}(f(z-k))}(r)\right)+S(f, r) \nonumber \\
&\leq 3e_k T_{f(z-k)}(r)+S(f, r),\label{eq:rightnow}
\end{align}
for $r$ outside of some set $E_1$ of finite measure, depending on $R$ and $k$. 
On the other hand, $f(z)=R_{z}\circ \cdots \circ R_{z-k+1}(f(z-k))$, and so Lemma~\ref{lem:valiron} gives
\begin{multline}\label{eq:degup}T_{f}(r)=\deg(R_{z}\circ \cdots \circ R_{z-k+1})T_{f(z-k)}x(r)+S(f, r)\\=\deg_w(R)^{k}T_{f(z-k)}(r)+S(f, r),\end{multline}
since the $R_{z-j}(w)$ all have the same degree in $w$ and (for $j\geq 0$) all have coefficients in $K_\phi$.
Thus we have from~\eqref{eq:bigk}, \eqref{eq:rightnow}, and~\eqref{eq:degup} that
\begin{align*}
m_{f}(r)&\leq \frac{\epsilon \deg_w(R)^k}{2}T_{f(z-k)}(r)+S(f, r)\\
& = \frac{\epsilon}{2}T_f(r)+S(f, r)\\
&\leq \epsilon T_f(r)
%
	\end{align*}
for $r$ sufficiently large and not in some set of finite length. This completes the proof in the case $a=\infty$.

Now let $a\in K_f$ be arbitrary, but ordinary for $R_z(w)$, suppose that $a(z+1)\in K_f$, and set
\[\mu_z(w)=\frac{1}{w}+a_z,\text{ so } \mu_z^{-1}(w)=\frac{1}{w-a_z}.\]
By definition, $m_{f}(r; a)=m_{\mu_z^{-1}(f)}(r)$. Let $S_z(w)=\mu_{z+1}^{-1}\circ R_{z}\circ \mu_z(w)$, and $g=\mu_z^{-1}(f)$. Note that we then have
 $g(z+1)=S_z(g(z))$ from~\eqref{eq:eq}. We also have, from Lemma~\ref{lem:valiron} (or an appropriate moving-targets version of the first main theorem) that
 \[T_g(r)=T_f(r)+O(T_a(r)).\]
 We note also that if $R_z(w)$ has coefficients in $K_f$, then $S_z$ has coefficients in $K_f$. We may apply the previous case of the result to obtain
 \begin{align*}
m_f(r; a)&=m_g(r)\\ &\leq \frac{\epsilon}{2}T_g(r)+S(f, r)\\
 	&= \frac{\epsilon}{2}T_f(r)+S(f, r)\\
 	&\leq \epsilon T_f(r)
 \end{align*}
 for $r$ sufficiently large and outside of an exceptional set of finite length.
\end{proof}


\bibliography{nevan}

\begin{thebibliography}{10}

\bibitem{MR3650542}
Gary~G. Gundersen.
\newblock Research questions on meromorphic functions and complex differential equations.
\newblock {\em Comput. Methods Funct. Theory}, 17(2):195--209, 2017.

\bibitem{diff_fin}
Patrick Ingram.
\newblock Effective finiteness of solutions to certain differential and difference equations.
\newblock 2021.
\newblock \emph{Canadian Mathematical Bulletin}, to appear.

\bibitem{MR0298006}
Anatolii~Z. Mohon'ko.
\newblock The {N}evanlinna characteristics of certain meromorphic functions.
\newblock {\em Teor. Funkci\u{\i} Funkcional. Anal. i Prilo\v{z}en.}, (14):83--87, 1971.

\bibitem{MR1555200}
Rolf Nevanlinna.
\newblock Zur {T}heorie der {M}eromorphen {F}unktionen.
\newblock {\em Acta Math.}, 46(1-2):1--99, 1925.

\bibitem{MR2106974}
Min Ru and Eunjeong Yi.
\newblock Nevanlinna theory and iteration of rational maps.
\newblock {\em Math. Z.}, 249(1):125--138, 2005.

\bibitem{MR1240603}
Joseph~H. Silverman.
\newblock Integer points, {D}iophantine approximation, and iteration of rational maps.
\newblock {\em Duke Math. J.}, 71(3):793--829, 1993.

\bibitem{MR850619}
Norbert Steinmetz.
\newblock Eine {V}erallgemeinerung des zweiten {N}evanlinnaschen {H}auptsatzes.
\newblock {\em J. Reine Angew. Math.}, 368:134--141, 1986.

\bibitem{MR1504970}
Georges Valiron.
\newblock Sur la d\'{e}riv\'{e}e des fonctions alg\'{e}bro\"{\i}des.
\newblock {\em Bull. Soc. Math. France}, 59:17--39, 1931.

\bibitem{MR621536}
Niro Yanagihara.
\newblock Meromorphic solutions of some difference equations.
\newblock {\em Funkcial. Ekvac.}, 23(3):309--326, 1980.

\bibitem{MR634897}
Niro Yanagihara.
\newblock Meromorphic solutions of some difference equations. {II}.
\newblock {\em Funkcial. Ekvac.}, 24(1):113--124, 1981.

\end{thebibliography}
\bibliographystyle{plain}

\end{document}